\documentclass[11pt]{article}
\usepackage{a4}
\usepackage{amssymb}
\usepackage{amsmath}
\usepackage{amsthm}
\usepackage{amsfonts}
\usepackage{enumerate}
\usepackage[pagebackref,colorlinks,citecolor=blue,linkcolor=blue]{hyperref}

\usepackage{geometry}
\numberwithin{equation}{section}
\theoremstyle{plain}
\newtheorem{theorem}[equation]{Theorem}

\newtheorem{lemma}[equation]{Lemma}
\newtheorem{proposition}[equation]{Proposition}

\theoremstyle{definition}
\newtheorem{definition}[equation]{Definition}

\newtheorem{remark}[equation]{Remark}

\numberwithin{equation}{section}

\newcommand{\N}{{\mathbb N}}

\providecommand{\vint}[1]{\mathchoice
          {\mathop{\vrule width 5pt height 3 pt depth -2.5pt
                  \kern -9pt \kern 1pt\intop}\nolimits_{\kern -5pt{#1}}}
          {\mathop{\vrule width 5pt height 3 pt depth -2.6pt
                  \kern -6pt \intop}\nolimits_{\kern -3pt{#1}}}
          {\mathop{\vrule width 5pt height 3 pt depth -2.6pt
                  \kern -6pt \intop}\nolimits_{\kern -3pt{#1}}}
          {\mathop{\vrule width 5pt height 3 pt depth -2.6pt
                  \kern -6pt \intop}\nolimits_{\kern -3pt{#1}}}}

\newcommand{\eps}{\varepsilon}
\newcommand{\loc}{\mathrm{loc}}

\newcommand{\BV}{\mathrm{BV}}

\newcommand{\liploc}{\mathrm{Lip}_{\mathrm{loc}}}

\newcommand{\ch}{\text{\raise 1.3pt \hbox{$\chi$}\kern-0.2pt}}

\newcommand{\CC}{\mathcal{C}\mathcal{C}}

\DeclareMathOperator{\frm}{\mathfrak{m}}
\DeclareMathOperator{\dfrm}{d\mathfrak{m}}

\DeclareMathOperator{\Mod}{Mod}

\DeclareMathOperator{\sat}{sat}

\begin{document}
\title{A note on indecomposable sets of finite perimeter
\let\thefootnote\relax\footnotetext{
\hfill \break {\it Keywords\,}: set of finite perimeter, indecomposable set,
metric measure space, isotropic space, Poincar\'e inequality, doubling measure
}}
\author{Panu Lahti}
\maketitle

\begin{abstract}
Bonicatto--Pasqualetto--Rajala \cite{BPR} proved that a decomposition theorem for sets of finite perimeter
into indecomposable sets, known to hold in Euclidean spaces, holds also in complete metric spaces equipped with
a doubling measure, supporting a Poincar\'e inequality, and satisfying an \emph{isotropicity}
condition. We show that the last assumption can be removed.
\end{abstract}

\section{Introduction}

A set of finite perimeter $E$ in a metric measure space $X$ is said to be \emph{indecomposable}
if it cannot be written as the disjoint union of two non-negligible sets $F,G$ with
$P(E,X)=P(F,X)+P(G,X)$. This measure-theoretic notion is similar to the topological notion of
connectedness. Properties of indecomposable sets in Euclidean spaces were studied by
Ambrosio et al. \cite{ACMM}; in particular, they proved that a set of finite perimeter can always
be uniquely decomposed into indecomposable sets. 

The theory was generalized to metric measure spaces by
Bonicatto--Pasqualetto--Rajala \cite{BPR}. As is common in analysis on metric measure spaces,
they assumed the space $(X,d,\frm)$ to be complete, equipped with a doubling measure, and support
a $(1,1)$-Poincar\'e inequality.
Such a space is called a PI space;
we will give definitions in Section \ref{sec:preliminaries}.
Additionally, they assumed that the representation of perimeter by means of the Hausdorff measure
satisfies an \emph{isotropicity} condition. This condition was previously
considered e.g. in \cite[Section 7]{AMP} and it is satisfied in Euclidean as well
as various other PI spaces. However, it excludes some PI spaces from the theory, see \cite[Example 1.27]{BPR}.

The proofs in \cite{BPR} relied heavily on the isotropicity condition,
and the condition was even shown to be necessary for certain results in the theory,
but its necessity for the main decomposition theorem remained unclear.
In the present paper, we show that the
isotropicity assumption can be removed.
On the other hand, the assumption of a $(1,1)$-Poincar\'e inequality cannot be removed,
see \cite[Example 2.16]{BPR}.

The following decomposition theorem is \cite[Theorem 2.14]{BPR},
except that there isotropicity was assumed.

\begin{theorem}\label{thm:main decomposition}
	Let $(X,d,\frm)$ be a PI space. Let $E\subset X$ be a set of finite perimeter. Then there exists a unique
	(finite or countable) partition $\{E_i\}_{i\in I}$ of $E$ into indecomposable subsets of $X$
	such that $\frm(E_i)>0$ for every $i\in I$ and $P(E,X)=\sum_{i\in I}P(E_i,X)$, where
	uniqueness is in the $\frm$-a.e. sense. Moreover, the sets $\{E_i\}_{i\in I}$
	are maximal indecomposable sets, meaning that for any Borel set $F\subset E$ with
	$P(F,X)<\infty$ that is indecomposable there is a (unique) $i\in I$
	such that $\frm(F\setminus E_i)=0$.
\end{theorem}

First in Section \ref{sec:main tools} we prove some
results on sets of finite perimeter
using only elementary methods, mostly involving basic properties of upper gradients in metric 
spaces. 
These results may be also of some independent interest.

Then in Section \ref{sec:decomposition} we show how the proofs in \cite{BPR} can be modified
to obtain the decomposition theorem (Theorem \ref{thm:main decomposition}) without the isotropicity assumption.
We also give a similar modification for some results of \cite{BPR} concerning \emph{holes}.
Other results of \cite{BPR}, especially certain ones concerning \emph{simple sets},
are shown there to be false unless one assumes isotropicity or even a stronger
\emph{two-sided property}. Thus we will not discuss these parts of the theory.

The presentation of Section \ref{sec:main tools} is essentially self-contained, whereas
Section \ref{sec:decomposition} largely consists of describing modifications to \cite{BPR}, so the
interested reader is advised to read that paper first.\\

\textbf{Acknowledgments.}
The author wishes to thank Tapio Rajala for discussions and for checking the manuscript.

\section{Notation and definitions}\label{sec:preliminaries}

In this section we introduce the basic notation, definitions,
and assumptions that are employed in the paper.

Throughout this paper, $(X,d,\frm)$ is a metric space that is equip\-ped
with a metric $d$ and a Borel regular outer measure $\frm$.
When a property holds outside a set of $\frm$-measure zero, we say that it holds
almost everywhere, abbreviated a.e.
We say that $\frm$ satisfies
a doubling property if
there exists a constant $C_d\ge 1$ such that
\begin{equation}\label{eq:doubling}
0<\frm(B(x,2r))\le C_d\frm(B(x,r))<\infty
\end{equation}
for every ball $B(x,r):=\{y\in X\colon\,d(y,x)<r\}$, $r>0$; we understand balls to be open.
We will not always assume the doubling property but we will always assume that
$0<\frm(B)<\infty$ for every ball $B$.

We assume $X$ to be proper, meaning that closed and bounded sets are compact.
All functions defined on $X$ or its subsets will take values in $[-\infty,\infty]$.
Given an open set $W\subset X$, we define $L^1_{\loc}(W)$
to be the class
of functions $u$ on $W$
such that $u\in L^1(W')$ for every open $W'\Subset W$, where the latter notation means that
$\overline{W'}$ is a compact subset of $W$.
Other local spaces of functions are defined analogously.

By a curve we mean a rectifiable continuous mapping from a compact interval of the real line into $X$.
The length of a curve $\gamma$
is denoted by $\ell_{\gamma}$. We assume every curve to be parametrized
by arc-length, which can always be done (see e.g. \cite[Theorem~3.2]{Hj}).
A nonnegative Borel function $g$ on $X$ is an upper gradient 
of a function $u$
on $X$ if for all nonconstant curves $\gamma$, we have
\begin{equation}\label{eq:definition of upper gradient}
|u(x)-u(y)|\le \int_{\gamma} g\,ds:=\int_0^{\ell_{\gamma}} g(\gamma(s))\,ds,
\end{equation}
where $x$ and $y$ are the end points of $\gamma$.
We interpret $|u(x)-u(y)|=\infty$ whenever  
at least one of $|u(x)|$, $|u(y)|$ is infinite.
Upper gradients were originally introduced in \cite{HK}.

The $1$-modulus of a family of curves $\Gamma$ is defined by
\[
\Mod_{1}(\Gamma):=\inf\int_{X}\rho\, \dfrm,
\]
where the infimum is taken over all nonnegative Borel functions $\rho$
such that $\int_{\gamma}\rho\,ds\ge 1$ for every curve $\gamma\in\Gamma$.
A property is said to hold for $1$-almost every curve
if it fails only for a curve family with zero $1$-modulus. 
If $g$ is a nonnegative $\frm$-measurable function on $X$
and (\ref{eq:definition of upper gradient}) holds for $1$-almost every curve,
we say that $g$ is a $1$-weak upper gradient of $u$. 
By only considering curves $\gamma$ in a set $A\subset X$,
we can talk about a function $g$ being a ($1$-weak) upper gradient of $u$ in $A$.

We say that $X$ supports a $(1,1)$-Poincar\'e inequality
if there exist constants $C_P>0$ and $\lambda \ge 1$ such that for every
ball $B(x,r)$, every $u\in L^1_{\loc}(X)$,
and every upper gradient $g$ of $u$,
we have
\[
\vint{B(x,r)}|u-u_{B(x,r)}|\, \dfrm 
\le C_P r\vint{B(x,\lambda r)}g\dfrm,
\]
where 
\[
u_{B(x,r)}:=\vint{B(x,r)}u\dfrm :=\frac 1{\frm(B(x,r))}\int_{B(x,r)}u\dfrm.
\]

If $X$ is complete, the measure $\frm$ satisfies the doubling condition \eqref{eq:doubling},
and the space supports a $(1,1)$-Poincar\'e inequality,
then we say that $X=(X,d,\frm)$ is a PI space.

Let $W\subset X$ be an open set. We let
\[
\Vert u\Vert_{N^{1,1}(W)}:=\Vert u\Vert_{L^1(W)}+\inf \Vert g\Vert_{L^1(W)},
\]
where the infimum is taken over all $1$-weak upper gradients $g$ of $u$ in $W$.
Then we define the Newton-Sobolev space
\[
N^{1,1}(W):=\{u\colon \|u\|_{N^{1,1}(W)}<\infty\},
\]
which was first introduced in \cite{S}.

We understand Newton-Sobolev functions to be defined at every $x\in W$
(even though $\Vert \cdot\Vert_{N^{1,1}(W)}$ is then only a seminorm).
It is known that for any $u\in N_{\loc}^{1,1}(W)$ there exists a minimal $1$-weak
upper gradient of $u$ in $W$, always denoted by $g_{u}$, satisfying $g_{u}\le g$ 
a.e. in $W$ for every $1$-weak upper gradient $g\in L_{\loc}^{1}(W)$
of $u$ in $W$, see e.g. the monograph Bj\"orn--Bj\"orn \cite[Theorem 2.25]{BB}.

Let $u,v \in N_{\loc}^{1,1}(W)$.
Then we have $u+v,\min\{u,v\},\max\{u,v\}\in N^{1,1}_{\loc}(W)$, see
\cite[Theorem 1.20]{BB}.
It is easy to see that the minimal $1$-weak upper gradients satisfy
\begin{equation}\label{eq:sum upper gradient}
g_{u+v}\le g_u+g_v\quad\textrm{a.e.}
\end{equation}
Also, 
\begin{equation}\label{eq:upper gradient locality}
g_u=g_v \textrm{ a.e. in }\{x \in W\colon  u(x)=v(x)\}
\end{equation}
by \cite[Corollary 2.21]{BB}.
For the above set, we use the short-hand notation $\{u=v\}$.
It follows that
\begin{equation}\label{eq:max upper gradient}
g_{\max\{u,v\}}=g_{u}\ch_{\{u>v\}}+g_{v}\ch_{\{u\le v\}}.
\end{equation}
It also follows that
\begin{equation}\label{eq:min max upper gradient}
\begin{split}
g_{\min\{u,v\}}+g_{\max\{u,v\}}
&=g_u\ch_{\{u\le v\}}+g_v\ch_{\{u> v\}}+g_u\ch_{\{u> v\}}+g_v\ch_{\{u\le v\}}\\
&=g_u+g_v\quad\textrm{a.e.}
\end{split}
\end{equation}

If $V\subset W$ are open subsets of $X$ and $u\in N_{\loc}^{1,1}(W)$, then also $u\in N_{\loc}^{1,1}(V)$,
and we denote the minimal $1$-weak upper gradient of $u$ in $V$ by $g_{u,V}$.
Note that usually we denote briefly $g_u=g_{u,W}$.
Now by \cite[Lemma 2.23]{BB} we have
\begin{equation}\label{eq:minimal ug wrt open set}
g_{u,V}=g_{u,W}\quad \textrm{a.e. in }V.
\end{equation}

Finally we note that $\liploc(W)\subset N_{\loc}^{1,1}(W)$,
see e.g. \cite[Proposition 1.14]{BB}.

Next we present the definition and basic properties of functions
of bounded variation on metric spaces, following \cite{M}. See also e.g. \cite{AFP, Fed} for the classical 
theory in the Euclidean setting.
Given an open set $W\subset X$ and a function $u\in L^1_{\loc}(W)$,
we define the total variation of $u$ in $W$ by
\[
\|Du\|(W):=\inf\left\{\liminf_{i\to\infty}\int_W g_{u_i}\dfrm:\, u_i\in \liploc(W),\, u_i\to u\textrm{ in } L^1_{\loc}(W)\right\},
\]
where each $g_{u_i}$ is again the minimal $1$-weak upper gradient of $u_i$
in $W$.
(In \cite{M}, pointwise Lipschitz constants were used in place of upper gradients, but the theory
can be developed similarly with either definition.)
We say that a function $u\in L^1(W)$ is of bounded variation, 
and denote $u\in\BV(W)$, if $\|Du\|(W)<\infty$.

From the definition it follows easily that if
$\{u_i\}_{i=1}^{\infty}$ is a sequence for which
$u_i\to u$ in $L^1_{\loc}(W)$ as $i\to\infty$, then
\begin{equation}\label{eq:lower semicontinuity}
\Vert Du\Vert(W)\le \liminf_{i\to\infty}\Vert Du_i\Vert(W).
\end{equation}

For an arbitrary set $A\subset X$, we define
\begin{equation}\label{eq:tot var arbitrary set}
\|Du\|(A):=\inf\{\|Du\|(W):\, A\subset W,\,W\subset X
\text{ is open}\}.
\end{equation}
In general, we understand the expression $\Vert Du\Vert(A)<\infty$ to mean that
there exists some open set $W\supset A$ such that $u$ is defined in $W$ with $u\in L^1_{\loc}(W)$ and $\Vert Du\Vert(W)<\infty$.

If $u\in L^1_{\loc}(W)$ and $\Vert Du\Vert(W)<\infty$,
then $\|Du\|(\cdot)$ is
a Borel regular outer measure on $W$ by \cite[Theorem 3.4]{M}.
A set $E\subset X$ is said to have finite perimeter
in $W$ if $\|D\ch_E\|(W)<\infty$, where $\ch_E$ is the characteristic function of $E$.
The perimeter of $E$ in $W$ is also denoted by
\[
P(E,W):=\|D\ch_E\|(W).
\]
If $P(E,X)<\infty$, we say briefly that $E$ is a set of finite perimeter.

For any set $A\subset X$ and $0<R<\infty$, the restricted Hausdorff content
of codimension one is defined by
\[
\mathcal{H}_{R}(A):=\inf\left\{ \sum_{j\in I}
\frac{\frm(B(x_{j},r_{j}))}{r_{j}}:\,A\subset\bigcup_{j\in I}B(x_{j},r_{j}),\,r_{j}\le R\right\},
\]
where the infimum is taken over finite and countable index sets $I$.
The codimension one Hausdorff measure of $A\subset X$ is then defined by
\[
\mathcal{H}(A):=\lim_{R\rightarrow 0}\mathcal{H}_{R}(A).
\]
For any $E\subset X$, the measure-theoretic boundary $\partial^{*}E$ is defined
as the set of points $x\in X$
at which both $E$ and its complement have strictly positive upper density, i.e.
\[
\limsup_{r\to 0}\frac{\frm(B(x,r)\cap E)}{\frm(B(x,r))}>0\quad
\textrm{and}\quad\limsup_{r\to 0}\frac{\frm(B(x,r)\setminus E)}{\frm(B(x,r))}>0.
\]
Suppose $(X,d,\frm)$ is a PI space.
For an open set $W\subset X$ and a $\frm$-measurable set $E\subset X$ with $P(E,W)<\infty$,
 for any Borel set $A\subset W$ we have
\begin{equation}\label{eq:def of theta}
P(E,A)=\int_{\partial^{*}E\cap A}\theta_E\,d\mathcal H,
\end{equation}
where
$\theta_E\colon \partial^*E\cap W\to [\alpha,C_d]$
with $\alpha=\alpha(C_d,C_P,\lambda)>0$, see \cite[Theorem 5.3]{A1} 
and \cite[Theorem 4.6]{AMP}.

\begin{definition}\label{def:isotropicity}
Let $(X,d,\frm)$ be a PI space. We say that $(X,d,\frm)$ is \emph{isotropic} if for
any pair of sets $E,F\subset X$ of finite perimeter with $F\subset E$
it holds that
\[
\theta_F(x) = \theta_E(x)\quad \textrm{for }\mathcal H\textrm{-a.e. }x\in \partial^*F\cap\partial^*E.
\]
\end{definition}

We will never assume isotropicity, but we give the definition for comparison with \cite{BPR}.

In some of our results, we will assume $(X,d,\frm)$ to be a PI space,
but our standing assumptions are merely the following.\\

\emph{Throughout this paper, $(X,d,\frm)$ is a proper metric space equipped with the
	Borel regular outer measure $\frm$, such that $0<\frm(B)<\infty$ for every ball $B$.}

\section{Main tools}\label{sec:main tools}

In this section we prove the following two propositions. These are the main tools that will be needed
in order to prove the decomposition theorem without the assumption of the space being isotropic.
In these propositions, we do not even need to assume $(X,d,\frm)$ to be a PI space, though in the decomposition
theorem this will be necessary. 

\begin{proposition}\label{prop:main decomposition result}
	Let $F,G\subset X$ with $\frm(F\cap G)=0$,
	and let $D\subset X$ be an arbitrary set such that
	\[
	P(F\cup G,D)=P(F,D)+P(G,D)<\infty.
	\]
	Let $F'\subset F$ and $G'\subset G$ be two other sets
	with $P(F',D)+P(G',D)<\infty$.
	Then
	\[
	P(F'\cup G',D)=P(F',D)+P(G',D).
	\]
\end{proposition}

\begin{proposition}\label{prop:perimeter in small large set}
	Let $W'\subset W\subset X$ be two open sets.
	Suppose $F,G\subset X$ with $\frm(F\cap G)=0$ and
	\[
	P(F\cup G,W)=P(F,W)+P(G,W)<\infty.
	\]
	Then also
	\[
	P(F\cup G,W')=P(F,W')+P(G,W').
	\]
\end{proposition}

We start by reciting the following well known fact and its proof.

\begin{lemma}\label{lem:BV functions form algebra}
	Let $F,G\subset X$ and $D\subset X$ be arbitrary sets. Then
	\[
	P(F\cap G,D)+P(F\cup G,D)\le P(F,D)+P(G,D).
	\]
\end{lemma}
\begin{proof}
	We can assume that the right-hand side is finite. Let $\eps>0$.
	By the definition of the total variation \eqref{eq:tot var arbitrary set},
	we find an open set $W\supset D$ such that
	\begin{equation}\label{eq:choice of W and D}
	P(F,W)<P(F,D)+\eps\quad\textrm{and}\quad P(G,W)<P(G,D)+\eps.
	\end{equation}
	Take sequences
	of functions $\{u_i\}_{i=1}^{\infty}$ and $\{v_i\}_{i=1}^{\infty}$ in $\liploc(W)$
	such that $u_i\to \ch_F$ and $v_i\to \ch_G$ in $L^1_{\loc}(W)$, and
	\[
	\lim_{i\to\infty}\int_{W}g_{u_i}\dfrm= P(F,W)\quad\textrm{and}\quad
	\lim_{i\to\infty}\int_{W}g_{v_i}\dfrm= P(G,W).
	\]
	We also have $\min\{u_i,v_i\}\to \ch_{F\cap G}$ and
	$\max\{u_i,v_i\}\to \ch_{F\cup G}$ in $L^1_{\loc}(W)$. Thus we get
	\begin{align*}
	&P(F\cap G,D)+P(F\cup G,D)\\
	&\qquad\qquad\le P(F\cap G,W)+P(F\cup G,W)\\
	&\qquad\qquad\le \liminf_{i\to\infty}\int_{W} g_{\min\{u_i,v_i\}}\dfrm+
	\liminf_{i\to\infty}\int_{W} g_{\max\{u_i,v_i\}}\dfrm\\
	&\qquad\qquad\le \liminf_{i\to\infty}\left(\int_{W} g_{\min\{u_i,v_i\}}\dfrm+
	\int_{W} g_{\max\{u_i,v_i\}}\dfrm\right)\\
	&\qquad\qquad = \liminf_{i\to\infty}\left(\int_{W} g_{u_i}\dfrm+
	\int_{W} g_{v_i}\dfrm\right)\quad\textrm{by }\eqref{eq:min max upper gradient}\\
	&\qquad\qquad= P(F,W)+P(G,W)\\
	&\qquad\qquad\le  P(F,D)+P(G,D)+2\eps.
	\end{align*}
	Letting $\eps\to 0$, we get the result.
\end{proof}

For any sequence of $\frm$-measurable sets $\{E_i\}_{i=1}^{\infty}$, we have
that $\ch_{\bigcup_{i=1}^j E_i} \to \ch_{\bigcup_{i=1}^{\infty} E_i}$ in $L^1_{\loc}(X)$ as $j\to\infty$, and so
by using \eqref{eq:lower semicontinuity} and then Lemma \ref{lem:BV functions form algebra}, we get
\begin{equation}\label{eq:countable subadditivity}
P\Big(\bigcup_{i=1}^{\infty} E_i,X\Big)\le \liminf_{j\to\infty} P\Big(\bigcup_{i=1}^{j} E_i,X\Big)
\le \sum_{i=1}^{\infty} P(E_i,X).
\end{equation}

Next we give the following result concerning approximating sequences that are
``almost optimal'' for a given set of finite perimeter.

\begin{lemma}\label{lem:almost optimal sequences}
Let $E'\subset E\subset X$ be two sets that have finite perimeter in an
open set $W\subset X$. Let $\eps\ge 0$ and
let $\{w_i\}_{i=1}^{\infty}$ be a sequence in $\liploc(W)$
such that $w_i\to \ch_E$ in $L^1_{\loc}(W)$ and
\[
\liminf_{i\to\infty}\int_{W}g_{w_i}\dfrm\le P(E,W)+\eps.
\] 
Then there exists a sequence $\{w_i'\}_{i=1}^{\infty}$ in $\liploc(W)$
such that $w_i'\le w_i$, $w_i'\to \ch_{E'}$ in $L^1_{\loc}(W)$, and
\[
\liminf_{i\to\infty}\int_{W}g_{w_i'}\dfrm \le P(E',W)+\eps.
\] 
\end{lemma}

\begin{proof}
	We choose a sequence $\{v_i\}_{i=1}^{\infty}$ in $\liploc(W)$
	such that $v_i\to \ch_{E'}$ in $L^1_{\loc}(W)$ and
	\begin{equation}\label{eq:choice of vi for W}
	\lim_{i\to\infty}\int_{W}g_{v_i}\dfrm= P(E',W).
	\end{equation}
Note that
\[
\max\{w_i,v_i\}\to \ch_E\ \textrm{ in }L^1_{\loc}(W),
\]
and so
\begin{equation}\label{eq:E' and E}
P(E,W)\le \liminf_{i\to\infty}\int_W g_{\max\{w_i,v_i\}}\dfrm.
\end{equation}
By assumption and by \eqref{eq:choice of vi for W}, we have
\begin{align*}
&P(E,W)+P(E',W)\\
&\qquad \ge \liminf_{i\to\infty}\left[\int_W g_{w_i}\dfrm+\int_W g_{v_i}\dfrm\right]-\eps\\
&\qquad =\liminf_{i\to\infty}\left[\int_W g_{\min\{w_i,v_i\}}\dfrm+
\int_W g_{\max\{w_i,v_i\}}\dfrm\right]-\eps\quad\textrm{by }
\eqref{eq:min max upper gradient}\\
&\qquad \ge \liminf_{i\to\infty}\int_W g_{\min\{w_i,v_i\}}\dfrm+P(E,W)-\eps\quad\textrm{by }
\eqref{eq:E' and E}.
\end{align*}
It follows that
\[
P(E',W)\ge \lim_{i\to\infty}\int_W g_{\min\{w_i,v_i\}}\dfrm-\eps.
\]
Now $w_i':=\min\{w_i,v_i\}\to \ch_{E'}$ in $L^1_{\loc}(W)$, giving the result.
\end{proof}

Proposition \ref{prop:main decomposition result} will follow almost directly from the following version
that considers open sets $W\subset X$.

\begin{lemma}\label{lem:open version}
	Let $W\subset X$ be an open set,
	let $F,G\subset X$ have finite perimeter in $W$ with $\frm(F\cap G)=0$,
	let $\eps\ge 0$, and
	suppose that
	\[
	P(F\cup G,W)+\eps\ge P(F,W)+P(G,W).
	\]
	Let $F'\subset F$ and $G'\subset G$ be two other sets with finite perimeter in $W$.
	Then also
	\[
	P(F'\cup G',W)+2\eps \ge P(F',W)+P(G',W).
	\]
\end{lemma}

\begin{proof}
	Choose sequences $\{u_i\}_{i=1}^{\infty}$ and $\{v_i\}_{i=1}^{\infty}$ in $\liploc(W)$
	such that $u_i\to \ch_F$ and $v_i\to \ch_{G}$ in $L^1_{\loc}(W)$, and
	\[
	\lim_{i\to\infty}\int_{W}g_{u_i}\dfrm = P(F,W)
	\quad\textrm{and}\quad 
	\lim_{i\to\infty}\int_{W}g_{v_i}\dfrm = P(G,W).
	\]
	Passing to subsequences (not relabeled), we can also get 
	$u_i\to \ch_F$ and $v_i\to \ch_{G}$ a.e. in $W$.
	We have
	\begin{equation}\label{eq:max is almost optimal}
	\begin{split}
	\limsup_{i\to\infty}\int_{W} g_{\max\{u_i,v_i\}}\dfrm
	&\le \limsup_{i\to\infty}\left[\int_{W} g_{u_i}\dfrm+\int_{W} g_{v_i}\dfrm\right]
	\quad\textrm{by }\eqref{eq:max upper gradient}\\
	&= P(F,W)+P(G,W)\\
	&\le P(F\cup G,W)+\eps\quad\textrm{by assumption}.
	\end{split}
\end{equation}
	We also have
\begin{align*}
&\lim_{i\to\infty}\left[\int_{W} g_{u_i}\dfrm+\int_{W} g_{v_i}\dfrm\right]\\
&\qquad = P(F,W)+P(G,W)\\
&\qquad \le P(F\cup G,W)+\eps\\
&\qquad \le \liminf_{i\to\infty}\int_{W} g_{\max\{u_i,v_i\}}\dfrm+\eps
\quad\textrm{since }\max\{u_i,v_i\}\to \ch_{F\cup G}\textrm{ in }L^1_{\loc}(W)\\
&\qquad = \liminf_{i\to\infty}\left[\int_{\{u_i>v_i\}} g_{u_i}\dfrm+\int_{\{u_i\le v_i\}} g_{v_i}\dfrm\right]+\eps
\quad \textrm{by }\eqref{eq:max upper gradient}.
\end{align*}
Thus necessarily (note that $g_{u_i}$ are $g_{v_i}$ are still the minimal $1$-weak upper gradients
in $W$, even though we integrate over a smaller set)
\[
\limsup_{i\to\infty}\left[\int_{\{u_i\le v_i\}} g_{u_i}\dfrm+
\int_{\{u_i> v_i\}} g_{v_i}\dfrm\right]\le \eps.
\]
Thus by the analog of \eqref{eq:max upper gradient} for $\min$, we get
\begin{equation}\label{eq:min g is small}
\begin{split}
\limsup_{i\to\infty}\int_W g_{\min\{u_i,v_i\}}\dfrm
= \limsup_{i\to\infty}\left[\int_{\{u_i\le v_i\}} g_{u_i}\dfrm+\int_{\{u_i> v_i\}} g_{v_i}\dfrm\right]
\le \eps.
\end{split}
\end{equation}
Since  $\frm(F\cap G)=0$, the functions
\[
w_i:=\max\{u_i,v_i\}-\min\{u_i,v_i\}
\]
still converge to $\ch_{F\cup G}$ in $L_{\loc}^1(W)$, with the following ``almost optimality'':
\begin{equation}\label{eq:wi almost optimality}
\begin{split}
\limsup_{i\to\infty}\int_{W}g_{w_i}\dfrm
&\le \limsup_{i\to\infty}\int_{W}g_{\max\{u_i,v_i\}}\dfrm+\limsup_{i\to\infty}\int_{W}g_{\min\{u_i,v_i\}}\dfrm
\quad\textrm{by }\eqref{eq:sum upper gradient}\\
&\le P(F\cup G,W)+2\eps\quad\textrm{by }\eqref{eq:max is almost optimal}\textrm{ and }\eqref{eq:min g is small}.
\end{split}
\end{equation}
Recall that $u_i,v_i\in\liploc(W)$ and so also $w_i\in\liploc(W)$.
Each $\{w_i>0\}$ consists of two disjoint open sets $F_i:=\{u_i>v_i\}$ and $G_i:=\{v_i>u_i\}$.
Since we had $u_i\to \ch_F$ and $v_i\to \ch_{G}$ a.e. in $W$, it follows that
\begin{equation}\label{eq:Fi WG}
\ch_{F_i}\to 0 \textrm{ a.e. in }G\textrm{ and }\ch_{G_i}\to 0 \textrm{ a.e. in }F.
\end{equation}
By Lemma \ref{lem:BV functions form algebra}, we have $P(F'\cup G',W)<\infty$.
Now by Lemma \ref{lem:almost optimal sequences} we find a sequence $\{w_i'\}_{i=1}^{\infty}$
in $\liploc(W)$ with $w_i'\to \ch_{F'\cup G'}$ in $L^1_{\loc}(W)$, $w_i'\le w_i$, and
\begin{equation}\label{eq:wi prime optimality}
\liminf_{i\to\infty}\int_{W}g_{w_i'}\dfrm\le P(F'\cup G',W)+2\eps.
\end{equation}
Passing to a subsequence (not relabeled), we can also get $w_i'\to \ch_{F'\cup G'}$ a.e. in $W$.
By truncating, we can assume that $w_i'\ge 0$
(then we still have $w_i'\le \max\{0,w_i\}$).
Now each set $\{w_i'>0\}\subset \{w_i>0\}$ is also contained in the union of the disjoint open sets $F_i$ and $G_i$.
It follows that $w'_i=w'_i \ch_{F_i} + w'_i \ch_{G_i}$.
We have $w'_i=0$ on $\partial F_i\cap W$, and so $w'_i \ch_{F_i}$ is in $\liploc(W)\subset N^{1,1}_{\loc}(W)$.
Similarly, $w'_i \ch_{G_i}$ is in $\liploc(W)$.
By \eqref{eq:upper gradient locality}, now
\begin{equation}\label{eq:gwi formula}
g_{w_i'}=g_{w_i'}\ch_{F_i}+g_{w_i'}\ch_{G_i}=g_{ w_i'\ch_{F_i}}+g_{ w_i'\ch_{G_i}}\quad\textrm{ a.e. in }W.
\end{equation}
Moreover, since $w_i'\to \ch_{F'\cup G'}$ a.e. in $W$ and using \eqref{eq:Fi WG}, we get
\begin{equation}\label{eq:wi char function convergence}
w'_i \ch_{F_i}\to \ch_{F'}\textrm{ a.e. in }W\quad 
\textrm{and similarly}\quad w'_i \ch_{G_i}\to\ch_{G'}\textrm{ a.e. in }W.
\end{equation}
By Lebesgue's dominated convergence theorem, we have these convergences also in $L_{\loc}^1(W)$.
It now follows that
\begin{align*}
P(F',W)+P(G',W)
&\le \liminf_{i\to\infty}\left[\int_{W} g_{w_i' \ch_{F_i} }\dfrm +\int_{W} g_{ w_i'\ch_{G_i}}\dfrm \right]\\
&=\liminf_{i\to\infty}\int_W g_{w_i'}\dfrm\quad
\textrm{by }\eqref{eq:gwi formula}\\
&\le P(F'\cup G',W)+2\eps\quad\textrm{by }\eqref{eq:wi prime optimality}.
\end{align*}
\end{proof}

\begin{proof}[Proof of Proposition \ref{prop:main decomposition result}.]
By Lemma \ref{lem:BV functions form algebra}, we know that $P(F'\cup G',D)<\infty$.
Let $\eps>0$.
By the definition of the total variation \eqref{eq:tot var arbitrary set}, there exists an open set
$W\supset D$ with
\begin{equation}\label{eq:choice of W 2}
	P(F'\cup G',W)\le P(F'\cup G',D)+\eps
\end{equation}
as well as
$P(F',W)<\infty$, $P(G',W)<\infty$, and
\begin{equation}\label{eq:choice of W 1}
	P(F,W) < P(F,D)+\eps/2\quad\textrm{and}\quad P(G,W) < P(G,D)+\eps/2.
\end{equation}
Then
\begin{align*}
P(F\cup G,W)
&\ge P(F\cup G,D)\\
&= P(F,D)+P(G,D)\quad\textrm{by assumption}\\
&\ge P(F,W)+P(G,W)-\eps\quad\textrm{by }\eqref{eq:choice of W 1}
\end{align*}	
and so
\begin{align*}
P(F'\cup G',D)
&\ge P(F'\cup G',W)-\eps\quad\textrm{by }\eqref{eq:choice of W 2}\\
&\ge P(F',W)+P(G',W)-2\eps\quad\textrm{by Lemma }\ref{lem:open version}\\
&\ge P(F',D)+P(G',D)-2\eps.
\end{align*}
Letting $\eps\to 0$, we get $P(F'\cup G',D)\ge P(F',D)+P(G',D)$. Since the opposite inequality always holds
by Lemma \ref{lem:BV functions form algebra}, we get the result.
\end{proof}

\begin{proof}[Proof of Proposition \ref{prop:perimeter in small large set}]
Take sequences $\{u_i\}_{i=1}^{\infty}$ and $\{v_i\}_{i=1}^{\infty}$
in $\liploc(W)$ such that
$u_i\to \ch_F$ in $L^1_{\loc}(W)$ and $v_i\to \ch_G$ in $L^1_{\loc}(W)$, as well as
\[
\lim_{i\to\infty}\int_{W} g_{u_i}\dfrm=P(F,W)\quad\textrm{and}\quad 
\lim_{i\to\infty}\int_{W} g_{v_i}\dfrm=P(G,W).
\]
Using \eqref{eq:max is almost optimal} with $\eps=0$, we get
\[
	P(F\cup G,W) \ge \limsup_{i\to\infty}\int_{W}g_{\max \{u_i,v_i\}}\dfrm,
\]
and then necessarily
\begin{equation}\label{eq:F cup G perimeter}
P(F\cup G,W) = \lim_{i\to\infty}\int_{W}g_{\max \{u_i,v_i\}}\dfrm.
\end{equation}
Define
\[
W_t':=\{x\in W'\colon d(x,X\setminus W')>t\},\quad t>0.
\]
Since $P(F\cup G,\cdot)$ is a Borel outer measure on $W$,
for all but at most countably many $t>0$ we have
\[
P(F\cup G,\partial W_t')=0\quad\textrm{and also}\quad \frm(\partial W_t')=0.
\]
For all such $t>0$, the fact that $\max \{u_i,v_i\}\to \ch_{F\cup G}$ in $L^1(W_t')$ implies
\[
P(F\cup G,W_t')\le \liminf_{i\to\infty}\int_{W_t'}g_{\max \{u_i,v_i\}}\dfrm
\]
(note that $g_{\max \{u_i,v_i\}}$ still denotes the minimal $1$-weak upper gradient in $W$;
recall \eqref{eq:minimal ug wrt open set})
and
\[
P(F\cup G,W\setminus \overline{W_t'})\le \liminf_{i\to\infty}\int_{W\setminus \overline{W_t'}}g_{\max \{u_i,v_i\}}\dfrm,
\]
and thirdly
\begin{align*}
P(F\cup G,W_t')+P(F\cup G,W\setminus \overline{W_t'})
&= P(F\cup G,W) \\
&=\lim_{i\to\infty}\int_{W} g_{\max \{u_i,v_i\}}\dfrm\quad
\textrm{by }\eqref{eq:F cup G perimeter}.
\end{align*}
Note that generally, if for nonnegative numbers $a,b,\{a_i\}_{i=1}^{\infty},\{b_i\}_{i=1}^{\infty}$ we have
\[
a\le \liminf_{i\to \infty}a_i\quad\textrm{and}\quad b\le \liminf_{i\to \infty}b_i
\quad\textrm{and}\quad a+b=\lim_{i\to\infty}(a_i+b_i),
\]
then necessarily $\lim_{i\to\infty}a_i=a$ and $\lim_{i\to\infty}b_i=b$.
Hence we get
\[
P(F\cup G,W_t')= \lim_{i\to\infty}\int_{W_t'}g_{\max \{u_i,v_i\}}\dfrm.
\]
By \eqref{eq:min g is small} with $\eps=0$, we have
\begin{equation}\label{eq:min g is small 2}
\limsup_{i\to\infty}\int_{W_t'}g_{\min\{u_i,v_i\}}\dfrm
\le \limsup_{i\to\infty}\int_{W}g_{\min\{u_i,v_i\}}\dfrm=0.
\end{equation}
Thus
\begin{align*}
P(F\cup G,W_t')
&=\lim_{i\to\infty}\int_{W_t'}g_{\max \{u_i,v_i\}}\dfrm\\
&=\lim_{i\to\infty}\left[\int_{W_t'}g_{u_i}\dfrm+\int_{W_t'}g_{v_i}\dfrm-
\int_{W_t'}g_{\min\{u_i,v_i\}}\dfrm \right]\quad\textrm{by }\eqref{eq:min max upper gradient}\\
&=\lim_{i\to\infty}\left[\int_{W_t'}g_{u_i}\dfrm+\int_{W_t'}g_{v_i}\dfrm\right]\quad
\textrm{by }\eqref{eq:min g is small 2}\\
&\ge P(F,W_t')+P(G,W_t')
\end{align*}
since $u_i\to \ch_F$ and $v_i\to \ch_G$ in $L^1_{\loc}(W_t')$.
Letting $t\to 0$, we get
\[
P(F\cup G,W')\ge P(F,W')+P(G,W').
\]
Since we also have the opposite inequality by Lemma \ref{lem:BV functions form algebra},
the result follows.
\end{proof}

\begin{remark}
In \cite[Section 2]{BPR}, results similar to this section were proved by relying on
the representation \eqref{eq:def of theta} of perimeter with respect to the Hausdorff measure,
as well as isotropicity (Definition \ref{def:isotropicity}).
Especially the first fact is very commonly used when studying BV functions,
see e.g. \cite{ACMM,AMP,KKST,L}, to mention just a few examples.
In this section, we have instead used only quite elementary methods, mostly
relying on the definition of perimeter and the basic properties of weak upper gradients,
which hold in much more general metric measure spaces than just PI spaces.
Thus our results and methods may be also of some independent interest. 
\end{remark}

\section{Decomposition theorem}\label{sec:decomposition}

In this section we show how to prove the decomposition theorem
as well as some other results of \cite{BPR} without assuming the space to be isotropic.
Many of the definitions and results below are directly from \cite{BPR}, apart from the fact that
$(X,d,\frm)$ is not assumed to be an isotropic PI space.

First we note a small technical point: until now we have understood
sets of finite perimeter to be $\frm$-measurable, but in \cite{BPR} a set of finite perimeter
is always understood to be Borel. For consistency, in this section we also adopt the latter as part of the
definition of sets of finite perimeter.

\begin{definition}
Let $E\subset X$ be a set of finite perimeter. Given any Borel set $D\subset X$, we say that
$E$ is \emph{decomposable} in $D$ provided there exists a partition $\{F,G\}$ of $E\cap D$
into sets of finite perimeter such that $\frm(F),\frm(G)>0$ and $P(E,D)=P(F,D)+P(G,D)$.
On the other hand, we say that $E$ is \emph{indecomposable in} $D$ if it is not
decomposable in $D$. For brevity, we say that $E$ is \emph{decomposable}
(resp. \emph{indecomposable}) provided it is decomposable in $X$ (resp. indecomposable in $X$).
\end{definition}

The following lemma is \cite[Lemma 2.8]{BPR},
except that there $(X,d,\frm)$ was assumed to be an isotropic PI space.

\begin{lemma}\label{lem:Lemma 2.8}
Let $E\subset X$ be a set
of finite perimeter and let $D\subset X$ be a Borel set. Suppose that $\{F,G\}$ is a Borel partition
of $E$ such that $P(E,D)=P(F,D)+P(G,D)$. Then $P(A,D)=P(A\cap F,D)+P(A\cap G,D)$
for every set $A\subset E$ of finite perimeter.
\end{lemma}
\begin{proof}
By Lemma \ref{lem:BV functions form algebra}, we have
$P(A\cap F,D)<\infty$ and $P(A\cap G,D)<\infty$.
Let $F':=A\cap F$ and $G':=A\cap G$.
Now the claim follows from Proposition \ref{prop:main decomposition result}.
\end{proof}

Next, we note that Corollary 2.9 of \cite{BPR} holds without the isotropicity
condition, since this condition is used in the proof only via the application of \cite[Lemma 2.8]{BPR}.

We also get the following

\begin{proposition}\label{prop:stability}
Let $E\subset X$ be a set of finite perimeter.
Let $\{E_n\}_{n=1}^{\infty}$ be an increasing sequence of indecomposable sets
such that $E=\bigcup_{n=1}^{\infty}E_n$. Then $E$ is an indecomposable set.
\end{proposition}
\begin{proof}
The proof is verbatim the same as that of \cite[Proposition 2.10]{BPR}, except that
instead of \cite[Lemma 2.8]{BPR} we apply Lemma \ref{lem:Lemma 2.8}.
\end{proof}

The following lemma is \cite[Lemma 2.11]{BPR}, except that there the sets $W',W$
were assumed to be Borel, and $(X,d,\frm)$ was assumed to be an isotropic PI space.

\begin{lemma}\label{lem:second lemma}
	Let $E\subset X$
	be a set of finite perimeter and let $W',W\subset X$ be open sets. Suppose that $E\subset W'\subset W$
	and that $E$ is indecomposable in $W'$. Then $E$ is indecomposable in $W$.
\end{lemma}
\begin{proof}
If $E$ is decomposable in $W$, then there exists a Borel partition
$\{F,G\}$ of $E$ such that $\frm(F)>0$ and $\frm(G)>0$, and
\[
P(E,W)=P(F,W)+P(G,W).
\]
Now by Proposition \ref{prop:perimeter in small large set}, we get
\[
P(E,W')=P(F,W')+P(G,W'),
\]
and thus $E$ is decomposable in $W'$.
\end{proof}

Next, \cite[Proposition 2.13]{BPR} also holds without isotropicity
(though the PI space assumption is now needed),
since this condition is used in the proof only via the application of \cite[Corollary 2.9]{BPR}. 

\begin{proof}[Proof of Theorem \ref{thm:main decomposition}]
We can follow the proof of \cite[Theorem 2.14]{BPR} almost verbatim,
since isotropicity is only used via the
application of \cite[Proposition 2.13]{BPR},
\cite[Lemma 2.11]{BPR} in open sets (given by Lemma \ref{lem:second lemma}),
and finally \cite[Proposition 2.10]{BPR} (Proposition \ref{prop:stability}).
\end{proof}

\begin{remark}
Example 2.16 in \cite{BPR} shows that in Theorem \ref{thm:main decomposition}, the assumption of a $(1,1)$-Poincar\'e inequality cannot be removed or even
replaced by any $(1,p)$-Poincar\'e inequality with $p>1$.
In this sense, our assumption that $(X,d,\frm)$ is a PI space is close to optimal.
\end{remark}

\begin{definition}
Let $(X,d,\frm)$ be a PI space. Let $E\subset X$ be a set of finite perimeter.
Then we denote by
\[
\CC^e(E):=\{E_i\}_{i\in I}
\]
the decomposition of $E$ given by Theorem \ref{thm:main decomposition}.
(By relabelling, we can assume that either
$I=\{1,\ldots,n\}$ or $I=\N$.)
The sets $E_i$ are called the \emph{essential connected components} of $E$.
\end{definition}

Note that with the above notation,
for any index set $J\subset I$ we have
\begin{align*}
P(E,X)
&\le P\Big(\bigcup_{i\in J}E_i,X\Big)
+P\Big(\bigcup_{i\in I\setminus J}E_i,X\Big)\quad
\textrm{by Lemma }\ref{lem:BV functions form algebra}\\
&\le \sum_{i\in J}P(E_i,X)+P\Big(\bigcup_{i\in I\setminus J}E_i,X\Big)
\quad\textrm{by }\eqref{eq:countable subadditivity}\\
&\le \sum_{i\in I}P(E_i,X)\quad\textrm{by }\eqref{eq:countable subadditivity}\\
&= P(E,X),
\end{align*}
and hence necessarily
\begin{equation}\label{eq:components subset}
P\Big(\bigcup_{i\in J}E_i,X\Big)=\sum_{i\in J}P(E_i,X).
\end{equation}
For any disjoint index sets $J_1,J_2\subset I$, this implies
\begin{equation}\label{eq:one component and the rest}
P\Big(\bigcup_{i\in J_1\cup J_2}E_i,X\Big)=P\Big(\bigcup_{i\in J_1}E_i,X\Big)+P\Big(\bigcup_{i\in J_2}E_i,X\Big).
\end{equation}

\begin{proposition}\label{prop:adding complement components}
	Let $(X,d,\frm)$ be a PI space.
	Let $E\subset X$ be an indecomposable set, and let $\CC^e(X\setminus E)=\{G_i\}_{i\in I}$.
	Then $E\cup \bigcup_{i\in J}G_i$ for any $J\subset I$ is also indecomposable, with
	\[
	P\Big(E\cup \bigcup_{i\in J}G_i,X\Big)\le P(E,X).
	\]
\end{proposition}
\begin{proof}
	Fix an index set $J\subset I$.
	We have either $J=\{i_1,\ldots,i_n\}$ or a sequence $J=\{i_1,i_2,\ldots\}$.
	First consider $G_{i_1}$.
	By Lemma \ref{lem:BV functions form algebra},
	$E\cup G_{i_1}$ has finite perimeter. Suppose $E\cup G_{i_1}$ is decomposable.
	Since $E$ and $G_{i_1}$ are indecomposable sets, necessarily one essential connected component
	(by its maximality property) of $E\cup G_{i_1}$ contains $E$ and another contains $G_{i_1}$
	(up to $\frm$-negligible sets).
	Thus the essential connected components of $E\cup G_{i_1}$ are $E$ and $G_{i_1}$, and so
	\begin{equation}\label{eq:E and G1}
	P(E\cup G_{i_1},X)=P(E,X)+P(G_{i_1},X).
	\end{equation}
	Now since $X\setminus E=G_{i_1}\cup \bigcup_{i\in I\setminus \{i_1\}}G_i$, we get
	\begin{align*}
	P(E,X)
	&= P\Big(G_{i_1}\cup \bigcup_{i\in I\setminus \{i_1\}}G_i,X\Big)\\
	&= P(G_{i_1},X)+P\Big(\bigcup_{i\in I\setminus \{i_1\}}G_i,X\Big)\quad\textrm{by }\eqref{eq:one component and the rest}\\
	&= P(G_{i_1},X)+P(E\cup G_{i_1},X)\\
	&= 2P(G_{i_1},X)+P(E,X)\quad\textrm{by }\eqref{eq:E and G1},
	\end{align*}
	implying $P(G_{i_1},X)=0$, a contradiction.
	Thus $E\cup G_{i_1}$ is indecomposable.
	Also,
	\begin{align*}
	P\Big(E\cup G_{i_1},X\Big)
	&=P\Big(\bigcup_{i\in I\setminus \{i_1\}}G_i,X\Big)\\
	&\le\sum_{i\in I\setminus \{i_1\}} P(G_i,X)\quad\textrm{by }\eqref{eq:countable subadditivity}\\
	&\le\sum_{i\in I} P(G_i,X)\\
	&=P\Big(\bigcup_{i\in I}G_i,X\Big)\\
	&=P(E,X).
	\end{align*}
	
	Now clearly
	$\{G_i\}_{i\in I\setminus \{i_1\}}$ is a partition of $\bigcup_{i\in I\setminus \{i_1\}}G_i$ into
	indecomposable subsets of $X$ with nonzero $\frm$-measure, and \eqref{eq:components subset} gives
	\[
	P\Big(\bigcup_{i\in I\setminus \{i_1\}}G_i,X\Big)=\sum_{i\in I\setminus \{i_1\}}P(G_i,X).
	\]
	Since such a decomposition is unique by
	Theorem \ref{thm:main decomposition}, necessarily
	$\CC^e(X\setminus (E\cup G_{i_1}))=\{G_{i}\}_{i\in I\setminus \{i_1\}}$.
	Thus we can repeat the first step and inductively obtain the result for any finite index set 
	$J$. Finally, if $J$ is infinite, the result is obtained by Proposition \ref{prop:stability}
	and the lower semicontinuity of perimeter.
\end{proof}

\begin{definition}
Let $(X,d,\frm)$ be a PI space such that $\frm(X)=\infty$. Let $E\subset X$ be an indecomposable set.
Then any essential connected component of $X\setminus E$ with finite $\frm$-measure is 
a \emph{hole} of $E$.
\end{definition}

\begin{definition}
Let $(X,d,\frm)$ be a PI space such that $\frm(X)=\infty$. Given an indecomposable set $E\subset X$,
we define its \emph{saturation} $\sat(E)$ as the union of $E$ and its holes.
We say that $E$ is \emph{saturated} provided that $\frm(E\Delta \sat(E))=0$.
\end{definition}

\begin{proposition}
Let $(X,d,\frm)$ be a PI space such that $\frm(X)=\infty$. Let $E\subset X$ be an indecomposable set.
Then the following properties hold:
\begin{enumerate}[(i)]
	\item Any hole of $E$ is saturated.
	\item The set $\sat(E)$ is indecomposable and saturated. In particular,
	$\sat(\sat(E))=\sat(E)$.
	\item It holds that $\mathcal H(\partial^*\sat(E)\setminus\partial^*E)=0$.
	In particular, $P(\sat(E),X)\le P(E,X)$.
	\item If $F\subset X$ is an indecomposable set with $\frm(E\setminus \sat (F))=0$,
	then $\frm(\sat(E)\setminus \sat(F))=0$.
\end{enumerate}
\end{proposition}
\begin{proof}
We can follow the proof of \cite[Proposition 3.12]{BPR} almost verbatim, with the following changes.
In the proof of (i), we need to apply Proposition \ref{prop:adding complement components} in place of
\cite[Proposition 2.18, Proposition 2.10]{BPR}.
In the proof of (iii), we note that
\cite[Remark 2.17]{BPR} is true also without the assumption of isotropicity,
and we get the second claim of (iii) from Proposition \ref{prop:adding complement components}.

\end{proof}

\noindent Address:\\

\noindent Academy of Mathematics and Systems Science,\\
Chinese Academy of Sciences,\\
Beijing 100190, PR China\\
E-mail: {\tt panulahti@amss.ac.cn}


\begin{thebibliography}{ACMM}

\bibitem{A1}L. Ambrosio,
\textit{Fine properties of sets of finite perimeter in doubling metric measure spaces},
Calculus of variations, nonsmooth analysis and related topics.
Set-Valued Anal. 10 (2002), no. 2-3, 111--128.

\bibitem{ACMM}L. Ambrosio, V. Caselles, S. Masnou, and J.-M. Morel,
\textit{Connected components of sets of finite perimeter and applications to image processing},
J. Eur. Math. Soc. (JEMS) 3 (2001), no. 1, 39--92.

\bibitem{AFP}L. Ambrosio, N. Fusco, and D. Pallara,
\textit{Functions of bounded variation and free discontinuity problems.}
Oxford Mathematical Monographs. The Clarendon Press, Oxford University Press, New York, 2000.

\bibitem{AMP}L. Ambrosio, M. Miranda, Jr., and D. Pallara,
\textit{Special functions of bounded variation in doubling metric measure spaces},
Calculus of variations: topics from the mathematical heritage of E. De Giorgi, 1--45,
Quad. Mat., 14, Dept. Math., Seconda Univ. Napoli, Caserta, 2004.

\bibitem{BB}A. Bj\"orn and J. Bj\"orn,
\textit{Nonlinear potential theory on metric spaces},
EMS Tracts in Mathematics, 17. European Mathematical Society (EMS), Z\"urich, 2011. xii+403 pp.

\bibitem{BPR}P. Bonicatto, E. Pasqualetto, and T. Rajala,
\textit{Indecomposable sets of finite perimeter in doubling metric measure spaces},
Calc. Var. Partial Differential Equations 59 (2020), no. 2, Paper No. 63, 39 pp.

\bibitem{Fed}H. Federer,
\textit{Geometric measure theory},
Die Grundlehren der mathematischen Wissenschaften, Band 153 Springer-Verlag New York Inc., New York 1969 xiv+676 pp.

\bibitem{Hj}P. Haj\l{}asz,
\textit{Sobolev spaces on metric-measure spaces.}
Heat kernels and analysis on manifolds, graphs, and metric spaces (Paris, 2002), 173--218,
Contemp. Math., 338, Amer. Math. Soc., Providence, RI, 2003.

\bibitem{HK}J. Heinonen and P. Koskela,
\textit{Quasiconformal maps in metric spaces with controlled geometry},
Acta Math. 181 (1998), no. 1, 1--61.

\bibitem{KKST}J. Kinnunen, R. Korte, N. Shanmugalingam, and H. Tuominen,
\textit{Pointwise properties of functions of bounded variation in metric spaces},
Rev. Mat. Complut. 27 (2014), no. 1, 41--67.

\bibitem{L}P. Lahti,
\textit{A new Federer-type characterization of sets of finite perimeter},
Arch. Ration. Mech. Anal. 236 (2020), no. 2, 801--838.

\bibitem{M}M.~Miranda, Jr.,
\textit{Functions of bounded variation on ``good'' metric spaces},
J. Math. Pures Appl. (9) 82  (2003),  no. 8, 975--1004.

\bibitem{S}N. Shanmugalingam,
\textit{Newtonian spaces: An extension of Sobolev spaces to metric measure spaces},
Rev. Mat. Iberoamericana 16(2) (2000), 243--279.

\end{thebibliography}
\end{document}